\documentclass[12pt]{amsart}
\usepackage{etex}
\usepackage[english]{babel}
\usepackage{latexsym}
\usepackage{amsbsy}
\usepackage{amssymb}
\usepackage{amsfonts, amsmath}
\usepackage{mathtools}
\usepackage{csquotes}
\usepackage{pst-tree}
\usepackage{multirow}
\usepackage{blkarray}
\usepackage{tabularx}
\usepackage{arydshln,leftidx,mathtools}
\usepackage{ifthen}
\usepackage{tikz}
\usepackage{tikz-cd}
\usepackage{float}
\usepackage{picins}
\usepackage{amsbsy}
\usetikzlibrary{matrix}
\usetikzlibrary{calc}
\usetikzlibrary{fit}
\usepackage{amsmath,amsfonts,amssymb}
\usepackage{mathrsfs,eurosym}
\usepackage{pst-plot,pst-eucl}
\usepackage{amsfonts}
\usepackage{amssymb}
\usepackage{amsmath}

\usepackage{latexsym}
\usepackage{amsbsy}
\usepackage{alltt}
\usepackage{fontenc}
\usepackage{newlfont}
\usepackage{xlop}
\usepackage{graphicx}
\usepackage{yhmath}
\usepackage{mathdots}

\usepackage{MnSymbol}
\usepackage{amsthm}
\usepackage{url}
\usepackage{enumerate}

\newtheorem{thrm}{Theorem}[section]
\newtheorem{lem}[thrm]{Lemma}
\newtheorem{prop}[thrm]{Proposition}

\theoremstyle{definition}

\newtheorem{remark}[thrm]{Remark}
\newtheorem{example}[thrm]{Example}
\numberwithin{equation}{section}

\usepackage{cite}

\numberwithin{equation}{section}

\newcommand{\e}{{\mathrm{e}}}

\title{$\mathcal{P}$-canonical forms and complete inverses}
\author{\bf M. MOU\c{C}OUF}
\date{}


\begin{document}
\maketitle
\begin{center}
{\footnotesize Department of Mathematics, Faculty of Science, Chouaib Doukkali University,\\Morocco\\
Email: moucouf@hotmail.com}
\end{center}
\begin{abstract}
This paper describes a new kind of inverse for elements in associative ring, that is the complete inverse, as the unique solution of a certain set of equations. This inverse exists for an element $a$ if and only if the Drazin inverse of $a$ exists. We also show that by plugging in $-k$ for $k$ in the $\mathcal{P}$-canonical form of a square matrix $A$, we get the $\mathcal{P}$-canonical form of the complete inverse of $A$.
\end{abstract}
%
\section{Introduction}
\label{sec1}
%
Let $R$ be an associative ring. A generalized inverse of an element $a$ of $R$ is an element that has some useful inverse properties and it is defined by the equation $aga=a$ where if $a$ is invertible we have $g=a^{-1}$. It is shown that the most natural generalized inverse is the Drazin inverse. Recall that the algebraic definition of the Drazin inverse $a_{d}$ of an element $a$ of $R$, for some positive integer $p$, is ~\cite{Draz}
$$a_{d}aa_{d}=a_{d}, \,\,aa_{d}=a_{d}a,\,\,a^{p}a_{d}a=a^{p}$$
In this paper we consider a new generalization of the notion of \enquote{inverse}, the idea is to start from the relationship $aga-gag=a-g$ which is consistent with the formula $aa^{-1}a-a^{-1}aa^{-1}=a-a^{-1}$. More precisely, we define the complete inverse $a_{c}$ of $a$ to be the unique element, if it exists, satisfying the equations
$$aa_{c}a-a_{c}aa_{c}=a-a_{c}, \,\, aa_{c}=a_{c}a,\,\,a^{p}a_{c}a=a^{p},$$
for some positive integer $p$. Then we give some properties of this inverse.
\\Next we prove that, in the case of a matrix $A$ over a field $F$ by plugging in $-k$ for $k$ in the $\mathcal{P}$-canonical form of $A$ we get the $\mathcal{P}$-canonical form of the matrix $A_{c}$.
%
\section{The complete inverse of an element of $R$}
\label{sect:The inverse}
%
Let $a$ be an element of an associative ring $R$. Let $n$ be a given positive integer, and consider the following equations in the unknown $x\in R$
$$\begin{array}{rlll}
a^{n}xa&=&a^{n}, \hspace*{8pc} & (1^{n})\\
xax&=&x, \hspace*{8pc} & (3)\\
xax-axa&=&x-a, \hspace*{8pc} & (4)\\
ax&=&xa, \hspace*{8pc} & (5)
\end{array}
$$
Let $U$ be a subset of $\{1^{n}, 3, 4, 5\}$. We say that $x$ is a $U$-inverse of $a$ if $x$ satisfies equation $(i)$ for all $i\in U$. The $\{1^{n},3,5\}$-inverse of $a$ is unique whenever it exists and it is called the Drazin inverse $a_{d}$ of $a$ (see e.g. Theorem $1$ of~\cite{Draz}). In this section we show some relationships between elements of $R$ and their Drazin and complete inverses.
\begin{lem}\label{lem 10} Let $a$ be an element of an associative ring $R$ and $n$ be a given positive integer. Suppose that $x$ is the $\{1^{n},3,5\}$-inverse of $a$ and set $v=a^{2}x$. Then
\\1. $v$ is the $\{1^{n},3,5\}$-inverse of $x$.
\\2. $v^{m}=a^{m}$ for all positive integer $m\geq n$.
\end{lem}
\proof~
\\1. This result is well known and easy to see~\cite{Draz}.
\\2. It is easy to see that $(ax)^{k}=ax$ for all positive integer $k$. Hence
\begin{eqnarray*}
v^{m}&=&a^{m}(ax)^{m}\\
&=&a^{m+1}x\\
&=&a^{m}.
\end{eqnarray*}
\endproof
\begin{thrm}\label{ thm 10}
Let $a$ be an element of an associative ring $R$ and $n$ be a given positive integer. Suppose that $x$ is the $\{1^{n},3,5\}$-inverse of $a$ and set $z=a+x-v$ where $v=a^{2}x$. Then $z$ is the unique $\{1^{n},4,5\}$-inverse of $a$.
\end{thrm}
\proof~
\\It is clear that $z$ commutes with $a$. Further, we have
\begin{eqnarray*}
a^{n+1}z&=&a^{n+2}+a^{n+1}x-a^{n+1}v\\
&=&a^{n+2}+a^{n}-v^{n+2}\\
&=&a^{n},
\end{eqnarray*}
and hence $z$ is a $\{1^{n}\}$-inverse of $a$.
\\We also have $$aza=a^{3}+a^{2}x-a^{2}v$$ and
$$zaz=a^{3}+ax^{2}+av^{2}+2a^{2}x-2a^{2}v-2axv.$$
Then
\begin{eqnarray*}
zaz-aza&=&ax^{2}+av^{2}+a^{2}x-a^{2}v-2axv\\
&=&x+av^{2}+v-av^{2}-v\\
&=&x+a^{2}x-2axv\\
&=&x-v\\
&=&z-a.
\end{eqnarray*}
For the uniqueness, let $u$ be a $\{1^{n},4,5\}$-inverse of $a$. Observe first that $u$ and $z$ commute since $x$ commutes with every element of $R$ which commutes with $a$ (see Theorem $1$ of~\cite{Draz}). On the other hand, we have
$$a^{n+1}u=a^{n}\quad\text{and}\quad a^{n+1}z=a^{n}$$
Then
\begin{eqnarray}\label{eq,,1}
a^{n+1}(u-z)=0.
\end{eqnarray}
 If we substract the equations $$az^{2}-a^{2}z=z-a$$ and $$au^{2}-a^{2}u=u-a,$$ we obtain $$-(u-z)=a^{2}(u-z)-a(u-z)(u+z).$$ Hence $$-a^{n}(u-z)=a^{n+2}(u-z)-a^{n+1}(u-z)(u+z).$$
Thus $$a^{n}(u-z)=0,$$
and we have peeled off one factor of $a$ from~\ref{eq,,1} to get $a^{n}(u-z)=0$. Clearly this process can be repeated as often as desired to obtain $u=z$.
\endproof
\begin{prop} Let $a$ be an element of an associative ring $R$ and $n$ be a given positive integer. Suppose that $x$ is the $\{1^{n},3,5\}$-inverse of $a$ and set $z=a+x-v$ where $v=a^{2}x$. Then
\\1. For all positive integer $k$, $z^{k}$ is the $\{1^{n},4,5\}$-inverse of $a^{k}$.
\\2. $z^{m}=x^{m}$ for all positive integer $m\geq n$.
\end{prop}
\proof~
\\1. Since $x^{k}$ and $v^{k}$ are the $\{1^{n},3,5\}$-inverse of $a^{k}$ and the $\{1^{n},3,5\}$-inverse of $x$ respectively, it is sufficient to show that $z^{k}=a^{k}+x^{k}-v^{k}$. To prove this result, we will proceed by induction on $k$. If $k=1$, this is trivial. Let $k\geq 1$, and assume that $z^{k}=a^{k}+x^{k}-v^{k}$. We write
\begin{eqnarray*}
z^{k+1}&=&z^{k}z\\
&=&(a^{k}+x^{k}-v^{k})(a+x-v)\\
&=&a^{k+1}+x^{k+1}+v^{k+1}-(av^{k}+a^{k}v)+(a^{k}x-xv^{k})+(ax^{k}-vx^{k}).
\end{eqnarray*}
Furthermore, we have
\begin{eqnarray*}
av^{k}&=&a^{k+1}(ax)^{k}\\
&=&a^{k+1}(ax)^{k+1}\\
&=&v^{k+1},
\end{eqnarray*}
\begin{eqnarray*}
a^{k}v&=&a^{k+1}ax\\
&=&a^{k+1}(ax)^{k+1}\\
&=&v^{k+1},
\end{eqnarray*}
\begin{eqnarray*}
xv^{k}&=&x(a^{2}x)^{k}\\
&=&a^{k}(ax)^{k}x\\
&=&a^{k}(ax)x\\
&=&a^{k}x
\end{eqnarray*}
and
\begin{eqnarray*}
vx^{k}&=&a^{2}xx^{k}\\
&=&a^{2}x^{2}x^{k-1}\\
&=&axx^{k-1}\\
&=&ax^{k}.
\end{eqnarray*}
Then $$z^{k+1}=a^{k+1}+x^{k+1}-v^{k+1},$$ and the proof is complete.
\\2. Follows imediately from Lemma~\ref{lem 10} and the fact that $z_{m}=a^{m}+x^{m}-v^{m}$.
\endproof
\begin{prop} Let $a$ be an element of an associative ring $R$ and $n$ be a given positive integer. Suppose that $x$ is the $\{1^{n},3,5\}$-inverse of $a$ and set $z=a+x-v$ where $v=a^{2}x$. Then
\\1. $v$ is the $\{1^{n},3,5\}$-inverse of $z$ and the $\{1^{n},4,5\}$-inverse of $x$.
\\2. $a$ is the $\{1^{n},4,5\}$-inverse of $z$.
\end{prop}
\proof~
\\1. It is obvious that $v$ and $z$ commute. On the other hand, we have
\begin{eqnarray*}
zv^{2}&=&av^{2}+xv^{2}-v^{3}\\
&=&xv^{2}\\
&=&v
\end{eqnarray*}
 and
\begin{eqnarray*}
z^{n+1}v&=&a^{n+1}v+x^{n+1}v-v^{n+2}\\
&=&x^{n+1}v\\
&=&x^{n+1}v\\
&=&x^{n}\\
&=&a^{n}+x^{n}-v^{n}\\
&=&z^{n}.
\end{eqnarray*}
Hence $v$ is the $\{1^{n},3,5\}$-inverse of $z$. For the second assertion it is sufficient to show that $v$ is a $\{4\}$-inverse of $x$, which is obvious since
$$xv^{2}=v\quad\text{and}\quad x^{2}v=ax^{2}=x.$$
2. We have
\begin{eqnarray*}
z^{n+1}a&=&a^{n+2}+x^{n+1}a-v^{n+1}a\\
&=&x^{n+1}a\\
&=&x^{n}\\
&=&z^{n}.
\end{eqnarray*}
Then the assertion follows from the obvious fact that $a$ is a $\{4,5\}$-inverse of $z$.
\endproof
\begin{prop}\label{prop,3} Let $a$ be an element of an associative ring $R$ and $n$ be a given positive integer. Suppose that $x$ is the $\{1^{n},3,5\}$-inverse of $a$ and set $z=a+x-v$ where $v=a^{2}x$. Then
\\1. $a$ is invertible if and only if $z$ is invertible; and in this case, we have $z=x$ and $a=v$.
\\2. If $R$ is a finite-dimensional algebra then $x$ and $z$ exist for all element $a\in R$. Furthermore, $z$ is a polynomial of $a$ and $x$ is a polynomial of $z$.
\\3. If $\mathcal{C}_{R}$ is the set of all sequences $\pmb{s}=(s_{k})_{k\geqslant 0}$ over $R$, then an element $\pmb{s}$ of $\mathcal{C}_{R}$ is a $U$-inverse of an element $\pmb{u}$ of $\mathcal{C}_{R}$ if and only if the element $s_{k}$ of $R$ is a $U$-inverse of the element $u_{k}$ of $R$ for all $k\geqslant 0$.
\end{prop}
\proof~
\\1. Suppose that $a$ is invertible, then $a^{-1}=x$ and $v=a$, and hence $z=x$. The converse follows similarly from $6.$
\\2. The first assertion and the first part of the second assertion follows from the fact that every element $a\in R$, the $\{1^{n},3,5\}$-inverse of $a$ exists and lies in the subalgebra generated by $a$ (see Corollary $5$ of~\cite{Draz}). As for the second part of the second assertion, observe that since $v^{n}$ is the $\{1^{n},3,5\}$-inverse of $x^{n}$, $v^{n}$ is expressible as a polynomial in $x^{n}$, say $v^{n}=H(x^{n})$. Hence
\begin{eqnarray*}
z^{n+1}H(z^{n})&=&z^{n+1}H(x^{n})\\
&=&z^{n+1}v^{n}\\
&=&x^{n+1}a^{n}\\
&=&x
\end{eqnarray*}
Thus $x$ is a polynomial in $z$.
\\3. This follows from straightforward verification.
\endproof
\begin{remark} We note that $a$ has a $\{1^{n},3,5\}$-inverse if and only if it has a $\{1^{n},4,5\}$-inverse. In fact, the only if part is an obvious consequence of Theorem~\ref{ thm 10}, while for the if part follows from~\cite{Draz} (Theorem $4$) and the fact that $a$ is strongly $\pi$-regular in $R$ if it has a $\{1^{n},4,5\}$-inverse.
\end{remark}
\section{The complete inverse of matrices}
Now let us consider the case of matrices. It is well-known that if $A$ is a square matrix of index $p$ over a filed $F$, then the $\{1^{p},3,5\}$-inverse $A_{d}$ of $A$ exists and it is called the Drazin inverse of $A$. Moreover, if the minimal polynomial of $A$ is $m_{A}(X)=X^{q}+a_{q-1}X^{q-1}+\cdots+a_{p}X^{p}, a_{p}\neq0$, then $$A_{d}=(-1)^{p+1}a_{p}^{-p-1}A^{p}(A^{q-(p+1)}+a_{q-1}A^{q-1-(p+1)}+\cdots+a_{p+1}I)^{p+1}$$
(see Corollary $5$ and the proof of Theorem $4$ of \cite{Draz}); and if the Jordan canonical form of $A$, in the splitting field of $m_{A}(X)$ over $F$, has the form
$$A=P\begin{pmatrix} D&0\\ 0&N \end{pmatrix}P^{-1}$$
where $P$ is a nonsingular matrix, $D$ is a nonsingular matrix of order $r=\text{rank}(A^{p})$, and $N$ is a
nilpotent matrix that $N^{p}=0$, then the Drazin inverse of $A$ is
$$A_{d}=P\begin{pmatrix} D^{-1}&0\\ 0&0 \end{pmatrix}P^{-1}.$$
We have the following result
\begin{prop}
Let $A$ be a square matrix of index $p$ over $F$. Then, with the notation above, the matrix
$$
A_{c}=P\begin{pmatrix} D^{-1}&0\\ 0&N \end{pmatrix}P^{-1}.
$$
is the unique $\{1^{p},4,5\}$-inverse of $A$.
\end{prop}
\begin{proof}
Follows immediately from Theorem~\ref{ thm 10} and the fact that $A_{c}=A+A_{d}-(A_{d})_{d}$.
\end{proof}
\begin{remark}
It is clear that $A$ is nonsingular if and only if $A_{c}$ is nonsingular; and in this case we have $A_{c}=A_{d}=A^{-1}$. This is a particular result of Proposition~\ref{prop,3}(1).
\end{remark}
Recall that in~\cite{Mouc} we have considered the following sets
\begin{itemize}
\item $\mathcal{C}_{F}$ the $F$-algebra of all sequences over $F$.
\item $\mathcal{S}^{\ast}=\{\pmb{\lambda}=(\lambda^{k})_{k\geqslant 0}/\lambda\in F, \lambda\neq 0\}$ be the set of all nonzero geometric sequences.
\item $\mathcal{S}^{\circ}=\{\pmb{0}_{i}\in \mathcal{C}_{F}/i\in \mathbb{N}, \pmb{0}_{i}(k)=\delta_{ik}\}$.
\item $\mathcal{S}=\mathcal{S}^{\ast}\cup \mathcal{S}^{\circ}$.
\item $F_{\mathcal{S}}$ denotes the $F$-vector spaces spanned by $\mathcal{S}$. It is well known that $F_{\mathcal{S}}$ is the set of all linear recurrence sequences in $F$ whose carateristic polynomials are of the form $X^{m}P(X)$ where $P(X)$ is square-free and split with nonzero constant terms.
\item $F_{\mathcal{S}^{\ast}}$ denotes the $F$-vector spaces spanned by $\mathcal{S}^{\ast}$. It is well known that $F_{\mathcal{S}^{\ast}}$ is the set of all linear recurrence sequences in $F$ whose carateristic polynomials are square-free and split with nonzero constant terms.
\item $F_{\mathcal{S}^{\circ}}$ denotes the $F$-vector spaces spanned by $\mathcal{S}^{\circ}$. It is well known that $F_{\mathcal{S}^{\circ}}$ the set of all linear recurrence sequences in $F$ whose carateristic polynomials split and have zero as its only root.
\end{itemize}
and we have noted the following:
\begin{itemize}
\item $\mathcal{S}$, $\mathcal{S}^{\ast}$ and $\mathcal{S}^{\circ}$ are basis of $F_{\mathcal{S}}$, $F_{\mathcal{S}^{\ast}}$ and $F_{\mathcal{S}^{\circ}}$, respectively.
\item $F_{\mathcal{S}}$, $F_{\mathcal{S}^{\ast}}$ and $F_{\mathcal{S}^{\circ}}$ are subalgebras of $\mathcal{C}_{F}$. More precisely, the set $\mathcal{S}^{\ast}$ is a group, and hence $F_{\mathcal{S}^{\ast}}$ is exactly the group algebra of $\mathcal{S}^{\ast}$ over $F$.
\item $F_{\mathcal{S}^{\circ}}$ and $F_{\mathcal{S}^{\ast}}$ are supplementary vector spaces relative to $F_{\mathcal{S}}$, i.e., $F_{\mathcal{S}}=F_{\mathcal{S}^{\circ}}\oplus F_{\mathcal{S}^{\ast}}$.
\end{itemize}
Also we have considered the following automorphism
\begin{align*}
\theta: F_{\mathcal{S}^{\ast}} &\longrightarrow F_{\mathcal{S}^{\ast}}\\
\pmb{\lambda} &\longmapsto \pmb{\lambda}^{-1}
\end{align*}
Direct summing this map with any linear map
\begin{align*}
\Psi: F_{\mathcal{S}^{\circ}}\longrightarrow  F_{\mathcal{S}^{\circ}}
\end{align*}
induced The following linear map
 \begin{align*}
\theta_{\Psi}=\Psi\oplus \theta: F_{\mathcal{S}}=F_{\mathcal{S}^{\circ}}\oplus F_{\mathcal{S}^{\ast}} \longrightarrow F_{\mathcal{S}}
\end{align*}
In turn, this extends naturally to $$\mathbf{M_{q}}(F_{\mathcal{S}})_{\mathcal{T}}=
\mathbf{M_{q}}(F_{\mathcal{S}^{\circ}})\oplus\mathbf{M_{q}}(F_{\mathcal{S}^{\ast}})_{\mathcal{T}}$$
\begin{align*}
\widetilde{\theta}_{\Psi}: \mathbf{M_{q}}(F_{\mathcal{S}})_{\mathcal{T}} &\longrightarrow \mathbf{M_{q}}(F_{\mathcal{S}})_{\mathcal{T}}\\
\pmb{U} &\longmapsto \mathcal{V}_{0}\Psi(\pmb{0}_{0})+\cdots+\mathcal{V}_{n}\Psi(\pmb{0}_{n})+
\widetilde{\theta}(\mathcal{W}_{0})\widetilde{\Lambda_{0}}+\cdots+\widetilde{\theta}(\mathcal{W}_{m})\widetilde{\Lambda_{m}},
\end{align*}
where $$\mathcal{V}_{0}\pmb{0}_{0}+\cdots+\mathcal{V}_{n}\pmb{0}_{n}+\mathcal{W}_{0}\Lambda_{0}+\cdots+\mathcal{W}_{m}\Lambda_{m}$$
is the $\mathcal{P}$-cf of $\pmb{U}$, $\widetilde{\Lambda_{m}}$ is the sequence obtained by plugging in $-k$ for $k$ in $\Lambda_{m}$ and
\begin{align*}
\widetilde{\theta}: \mathbf{M_{q}}(F_{\mathcal{S}^{\ast}}) &\longrightarrow \mathbf{M_{q}}(F_{\mathcal{S}^{\ast}})\\
(A_{ij}) &\longmapsto (\theta(A_{ij}))
\end{align*}
In Theorem 3.9 of~\cite{Mouc}, we have taken $\Psi=0$ and we have proved in this case that, for all square matrix $A$, the Drazin inverse of the sequence $\pmb{A}=(A^{k})$ is $\pmb{A}_{d}=\widetilde{\theta}_{0}(\pmb{A})+\pmb{0}_{0}\pi_{0}$, where $\pi_{0}$ is the spectral projection of $A$ at $0$, which is the zero matrix if $A$ is of index $0$.
\\Now consider the case where $\Psi=id$ is the identity map and let us denote $\widetilde{\theta}_{1}=\widetilde{\theta}_{id}$.
\\With this notation we can state the following result.
\begin{thrm}\label{thm 535}
Let $A$ be a square matrix of index $t_{0}$ with coefficients in $F$ such that its characteristic polynomial splits over $F$. Then the complete inverse of the sequence $\pmb{A}$ is $\pmb{A}_{c}=\widetilde{\theta}_{1}(\pmb{A})$.
\end{thrm}
\proof~
\\Let $N(A)$ be the non-geometric part of $A$. Then we have
\begin{eqnarray*}
\pmb{A}_{D}&=&\widetilde{\theta}_{0}(\pmb{A})+\pmb{0}_{0}\pi_{0}\\
&=&\widetilde{\theta}(\pmb{A}-N(A))+\pmb{0}_{0}\pi_{0}
\end{eqnarray*}
and then
\begin{eqnarray*}
(\pmb{A}_{D})_{D}&=&\widetilde{\theta}_{0}(\widetilde{\theta}(\pmb{A}-N(A))+\pmb{0}_{0}\pi_{0})+\pmb{0}_{0}\pi_{0}\\
&=&\widetilde{\theta}\circ \widetilde{\theta}(\pmb{A}-N(A))+\pmb{0}_{0}\pi_{0}\\
&=&\pmb{A}-N(A)+\pmb{0}_{0}\pi_{0}.
\end{eqnarray*}
 Hence
\begin{eqnarray*}
\pmb{A}+\pmb{A}_{D}-(\pmb{A}_{D})_{D}&=&N(A)+\widetilde{\theta}(\pmb{A}-N(A))\\
&=&\widetilde{\theta}_{1}(\pmb{A}).
\end{eqnarray*}
Therefore, $\pmb{A}_{c}=\widetilde{\theta}_{1}(\pmb{A})$.
\endproof
\begin{remark} $\pmb{A}_{c}$ is obtained by simply plugging in $-k$ for $k$ in the $\mathcal{P}$-canonical form of $A$.
\end{remark}
\begin{example}\label{example 1}
Let $$A=\begin{pmatrix}
1&1&1&0\\1&1&1&-1\\0&0&-1&1\\0&0&1&-1
\end{pmatrix}$$
and
$$A(k)=\begin{pmatrix}
2^{-1+k}&2^{-1+k}&\frac{1}{16}2^{k}((-1)^{1+k}+5)&\frac{1}{16}2^{k}((-1)^{k}-1)\vspace*{0.1pc}\\
2^{-1+k}&2^{-1+k}&\frac{5}{16}2^{k}((-1)^{1+k}+1)&\frac{1}{16}2^{k}(5(-1)^{k}-1)\vspace*{0.1pc}\\
0&0&(-1)^{k}2^{-1+k}&(-1)^{1+k}2^{-1+k}\vspace*{0.1pc}\\
0&0&(-1)^{1+k}2^{-1+k}&(-1)^{k}2^{-1+k}
\end{pmatrix}.$$
Then we have
\begin{itemize}
\item The non-geometric part of $A$ is $(I_{4}-A(0))\pmb{0}_{0}+(A-A(1))\pmb{0}_{1}$.
\item The geometric part of $A$ is
$$\begin{pmatrix}
2^{-1}(2^{k})&2^{-1}(2^{k})&\frac{5}{16}(2^{k})-\frac{1}{16}((-2)^{k})&\frac{-1}{16}(2^{k})+\frac{1}{16}((-2)^{k})\vspace*{0.1pc}\\
2^{-1}(2^{k})&2^{-1}(2^{k})&\frac{5}{16}(2^{k})-\frac{5}{16}((-2)^{k})&\frac{-1}{16}(2^{k})+\frac{5}{16}((-2)^{k})\vspace*{0.1pc}\\
0&0&2^{-1}((-2)^{k})&-2^{-1}((-2)^{k})\vspace*{0.1pc}\\
0&0&-2^{-1}((-2)^{k})&2^{-1}((-2)^{k})
\end{pmatrix}$$
\item For all $k\geq 0$, $A_{c}^{k}=$
$$\begin{pmatrix}
2^{-1-k}+\frac{\pmb{0}_{0}(k)}{2}&2^{-1-k}-\frac{\pmb{0}_{0}(k)}{2}&\frac{1}{16}2^{-k}((-1)^{1+k}+5)-\frac{\pmb{0}_{0}(k)}{4}+
\frac{\pmb{0}_{1}(k)}{4}&
\frac{1}{16}2^{-k}((-1)^{k}-1)+\frac{5\pmb{0}_{1}(k)}{4}\vspace*{0.1pc}\\
2^{-1-k}-\frac{\pmb{0}_{0}(k)}{2}&2^{-1-k}+\frac{\pmb{0}_{0}(k)}{2}&\frac{5}{16}2^{-k}((-1)^{1+k}+1)-\frac{\pmb{0}_{1}(k)}{4}&
\frac{1}{16}2^{-k}(5(-1)^{k}-1)-\frac{\pmb{0}_{0}(k)}{4}-\frac{\pmb{0}_{1}(k)}{4}\vspace*{0.1pc}\\
0&0&(-1)^{k}2^{-1-k}+\frac{\pmb{0}_{0}(k)}{2}&(-1)^{1+k}2^{-1-k}+\frac{\pmb{0}_{0}(k)}{2}\vspace*{0.1pc}\\
0&0&(-1)^{1+k}2^{-1-k}+\frac{\pmb{0}_{0}(k)}{2}&(-1)^{k}2^{-1-k}+\frac{\pmb{0}_{0}(k)}{2}
\end{pmatrix}$$
\item In particular,
$$A_{c}=\begin{pmatrix}
\frac{1}{4}&\frac{1}{4}&\frac{7}{16}&\frac{19}{16}\vspace*{0.1pc}\\
\frac{1}{4}&\frac{1}{4}&\frac{1}{16}&\frac{-7}{16}\vspace*{0.1pc}\\
0&0&\frac{-1}{4}&\frac{1}{4}\vspace*{0.1pc}\\
0&0&\frac{1}{4}&\frac{-1}{4}
\end{pmatrix}.$$
\end{itemize}
\end{example}
\begin{example} Take $x=0$ in the example 3.5 of~\cite{Mouc} to have
\begin{itemize}
\item $$E=\begin{pmatrix}
2\sqrt{3}-10&2\sqrt{3}-23&\sqrt{3}-5\\4&\sqrt{3}+9&2\\
-2\sqrt{3}+2&-4\sqrt{3}+5&-\sqrt{3}+1
\end{pmatrix}$$
\item and
$$E(k)=
\begin{pmatrix}
e_{11}(k)&e_{12}(k)&e_{13}(k)\\e_{21}(k)&e_{22}(k)&e_{23}(k)\\e_{31}(k)&e_{32}(k)&e_{33}(k)
\end{pmatrix}$$
where
\begin{eqnarray*}
e_{11}(k)&=&2^{k+1}(\cos(\frac{k\pi}{6})-5\sin(\frac{k\pi}{6}))\\
e_{12}(k)&=&2^{k+1}(\cos(\frac{k\pi}{6})-\frac{23}{2}\sin(\frac{k\pi}{6}))\\
e_{13}(k)&=&2^{k}(\cos(\frac{k\pi}{6})-5\sin(\frac{k\pi}{6}))\\
e_{21}(k)&=&2^{k+2}\sin(\frac{k\pi}{6})\\
e_{22}(k)&=&2^{k}(\cos(\frac{k\pi}{6})+9\sin(\frac{k\pi}{6}))\\
e_{23}(k)&=&2^{k+1}\sin(\frac{k\pi}{6})\\
e_{31}(k)&=&-2^{k+1}(\cos(\frac{k\pi}{6})-\sin(\frac{k\pi}{6}))\\
e_{32}(k)&=&-2^{k+2}(\cos(\frac{k\pi}{6})-\frac{5}{4}\sin(\frac{k\pi}{6}))\\
e_{33}(k)&=&-2^{k}(\cos(\frac{k\pi}{6})-\sin(\frac{k\pi}{6}))
\end{eqnarray*}
\item We have, for all $k\geq1$,
$E^{k}=E(k)$ and
$$E(0)=\begin{pmatrix} 2&2&1\\0&1&0\\-2&-4&-1\end{pmatrix}$$
Then $E$ is singular with index one.
\item For all $k\geq1$, we have $E_{c}^{k}=E(-k)$.
\item In particular,
$$E_{c}=\begin{pmatrix} \frac{\sqrt{3}+5}{2}&\frac{2\sqrt{3}+23}{4}&\frac{\sqrt{3}+5}{4}\vspace*{0.1pc}\\
-1&\frac{\sqrt{3}-9}{4}&\frac{-1}{2}\vspace*{0.1pc}\\\frac{-\sqrt{3}-1}{2}&\frac{-4\sqrt{3}-5}{4}&
\frac{-\sqrt{3}-1}{4}\end{pmatrix}$$
\end{itemize}
\end{example}


\begin{thebibliography}{99}
\bibitem{Draz}  M. R. Drazin, \emph{Pseudo-inverses in associative rings and semigroups,}
Amer. Math. Monthly. {\bf 65}(1958), 506-514.
\bibitem{Mouc} M. Mouçouf, \emph{$\mathcal{P}$-canonical forms and Drazin inverses,} arXiv:2007.10199 [math.RA]
\end{thebibliography}
\end{document}